\documentclass{amsart}
\usepackage{amssymb,amsmath,amsthm}
\usepackage[foot]{amsaddr}
\title{A topos for continuous logic}
\author{Daniel Figueroa and Benno van den Berg} 
\date{\today}

\usepackage{ast2}

\usepackage{amsmath}
\usepackage{amsthm}
\usepackage{bbm}
\usepackage{breqn}
\usepackage{enumitem}
\usepackage[margin=1.7in]{geometry}
\usepackage{hyperref}
\usepackage{relsize}
\usepackage{tikz-cd}
\usepackage{xcolor}
\usepackage{multicol}
\usepackage[
    backend=biber,
    style=alphabetic-verb,
    sortlocale=en_GB,
    natbib=true,
    url=false,
    isbn=false,
    doi=true,
    eprint=true
    ]{biblatex}

\newtheorem{defn}{Definition}

\newcommand{\N}{\mathbb{N}}

\newcommand{\sq}{\sqsubseteq}

\newcommand{\ang}[1]{\langle #1 \rangle}
\newcommand{\brac}[1]{\llbracket #1 \rrbracket}

\newcommand{\bfp}{\mathbf{P}}
\newcommand{\bfq}{\mathbf{Q}}

\begin{document}

\begin{abstract} 
We suggest an ordering for the predicates in continuous logic so that the semantics of continuous logic can be formulated as a hyperdoctrine. We show that this hyperdoctrine can be embedded into the hyperdoctrine of subobjects of a suitable Grothendieck topos. For this embedding we use a simplification of the hyperdoctrine for continuous logic, whose category of equivalence relations is equivalent to the category of complete metric spaces and uniformly continuous maps. 
\end{abstract}

\maketitle

\section{Introduction} \label{section:intro}
\addcontentsline{toc}{section}{\nameref{section:intro}}

Continuous first-order logic is a modification of classical first-order logic introduced by Yaacov and Usvyatsov in \cite{yaacov_usvyatsov_2010} and resembles the earlier logic introduced by Chang and Keisler in \cite{chang_keisler_1966}. Generally speaking, in continuous logic the structures are complete metric spaces of finite diameter, say of diameter 1, terms are uniformly continuous functions between such metric spaces, and formulas are uniformly continuous functions into the unit interval. One of its uses is that it allows the study of elementary classes of complete metric structures that are not elementary in classical model theory. As a result continuous model theory has led to new applications in analysis. But also from the theoretical side, the model theory of continuous logic has been successful and many concepts from classical logic such as completeness and compactness can be given continuous counterparts \cite{yaacov_usvyatsov_2010}. 

In this paper we wish to analyse continuous logic from the point of view of categorical logic. We are not the first to do so; however, the earlier attempts we are aware of (such as \cite{albert2016metric} and \cite{cho2019categorical}) start by modifying the standard framework of categorical logic. In the background, Lawvere's idea to understand metric spaces as enriched categories has been very influential \cite{lawvere_1973}. In contrast, the analysis we give here does not rely on enriched category theory; instead, it uses conventional concepts from categorical logic, such as the logic of a topos and Lawvere's notion of a hyperdoctrine.

The starting point for our analysis is a new hyperdoctrine, which we call $\mathbf{CMT}$ for continuous model theory. The main aim of this paper is to embed this hyperdoctrine into the subobject hyperdoctrine of a Grothendieck topos. In doing so, we show that continuous logic is subsumed by the framework for categorical logic provided by topos theory. For this embedding we use a simplification of the hyperdoctrine $\mathbf{CMT}$, which we call $\mathbf{U}$. We also show that the category of equivalence relations over $\mathbf{U}$ is equivalent to the category of complete metric spaces and uniformly continuous maps. 

\section*{Notation}\label{section: notation}
\addcontentsline{toc}{section}{\nameref{section: notation}}

\begin{itemize}
    \item If $\mathcal{C}$ is a category and $X$ and $Y$ are objects in $\mathcal{C}$ then we write $X \times Y$ for the product of $X$ and $Y$ and $X + Y$ for the coproduct of $X$ and $Y$. If $f$ and $g$ are morphisms in $\mathcal{C}$, then we write $\ang{f,g}$ for the pairing of $f$ and $g$ and we write $[f,g]$ for the copairing of $f$ and $g$.
    \item If $f$ and $g$ are morphisms in a category with common codomain, we write $f^*(g)$ for the pullback of $g$ along $f$.
    \item If $S$ is a set, $s \in S$, and if there is a clear notion of an equivalence relation $\sim$ on $S$ we write $[s]$ for the equivalence class of $S$ by $\sim$ represented by $s$.
    \item Let $X$ and $Y$ be sets and let $u: X^n \rightarrow Y$ be a function. If $\phi_1, ... , \phi_n$ are functions $X \rightarrow Y$ then we write $u(\phi_1,...,\phi_n)$ for the function $X \rightarrow Y$ that takes $x \in X$ to  $u(\phi_1(x),...,\phi_n(x))$.
    \item By $[0,1]$ we denote the real unit interval and by $[0,1]^0$ we denote the set $\{0\}$.
    \item By ${\mathbf{Preorders}}$ we denote the category of preorders and order-preserving functions.
    By ${\mathbf{DistPrelattices}}$ we denote the category where objects are preorders whose poset reflections are distributive lattices and morphisms are order-preserving functions whose poset reflections are lattice homomorphisms. If $A$ is a preorder and $a, b \in A$, we write $a \simeq b$ and call $a$ and $b$ isomorphic if $a \leq b$ and $b \leq a$. More generally, $A$ and $B$ are objects and $f,g : A \rightarrow  B$ morphisms in any of these categories, we say that $f \leq g$ if $f(a) \leq_B g(a)$ for all $a \in A$, where $\leq_B$ is the ordering in $B$. We say that such morphisms $f$ and $g$ are isomorphic and write $f \simeq g$ if $f \leq g$ and $g \leq f$.
\end{itemize}
 
\section{Metric spaces and uniformities}

We start by recalling some basic definitions and facts about metric spaces and uniformities.

\begin{defi}{metricspace} A \emph{metric space} is a pair $(X,d)$ consisting of a set $X$ together with a function $d: X \times X \to \mathbb{R}_{\geq 0}$ satisfing the following three requirements for all $x, y, z \in X$:
    \begin{enumerate}
        \item $d(x,y) = 0$ if and only if $x = y$.
        \item $d(x,y) = d(y,x)$.
        \item $d(x, z) \leq d(x,y) + d(y,z)$.
    \end{enumerate}
The function $d$ is called the \emph{metric} and the real number $d(x,y)$ is called the \emph{distance} between $x$ and $y$. 
If requirement (i) is weakened to the statement that $d(x,x) = 0$ for all $x \in X$, we call the pair $(X,d)$ a \emph{pseudometric space} and $d$ a \emph{pseudometric}. 
\end{defi}
 
\begin{defi}{complmetricsp}
A sequence of elements $(x_n)_{n \in \mathbb{N}}$ in a (pseudo)metric space is a \emph{Cauchy sequence} if for each $\varepsilon \gt 0$ there is an element $K \in \mathbb{N}$ such that $d(x_n,x_m) \leq \varepsilon$ for all $n, m \geq K$. An element $x \in X$ is a \emph{limit} of a Cauchy sequence $(x_n)_{n \in \mathbb{N}}$ if for each $\varepsilon \gt 0$ there is a $K \in \mathbb{N}$ such that $d(x,x_n) \leq \varepsilon$ for all $n \geq K$; we will also say that the sequence $(x_n)_{n \in \N}$ converges to $x$ and write $(x_n)_{n \in \N} \to x$. A (pseudo)metric space is \emph{complete} if each Cauchy sequence has a limit.
\end{defi}

We wish to treat the (complete) (pseudo)metric spaces as objects in a category. For that we need to single out a special classes of morphisms, and the natural class for us, given the connection we wish to establish with continuous model theory, is the class of uniformly continuous functions.

\begin{defi}{uniformcontinuity}
    If $(X,d)$ and $(Y,d)$ are (pseudo)metric spaces, then a function $f: X \to Y$ is called \emph{uniformly continuous} if for each $\varepsilon \gt 0$ there is a $\delta \gt 0$ such that for all $x_0,x_1 \in X$ we have:
    \[ d(x_0,x_1) \leq \delta \Longrightarrow d(f(x_0), f(x_1)) \leq \varepsilon. \]
    We write $\mathbf{Met}$ for the category of metric spaces and uniformly continuous maps, ${\mathbf{pMet}}$ for the category of pseudometric spaces and uniformly continuous maps, and ${\mathbf{cMet}}$  for the category of complete metric spaces and uniformly continuous maps.
\end{defi}    

\begin{defi}{diameter}
    We will say that a pseudometric space $(X,d)$ \emph{has diameter at most 1} if $d(x,y) \leq 1$ for all $x, y \in X$. In this case we will regard the metric $d$ as a function $d: X \times X \to [0,1]$. We will write ${\mathbf{Met}}_1$, ${\mathbf{pMet}}_1$ and ${\mathbf{cMet}}_1$ for the subcategories consisting of the (ordinary, pseudo- or complete) metric spaces having diameter at most 1.
\end{defi}

\begin{rema}{distance1norestriction} 
Note that the inclusion functor $I: {\mathbf{Met}}_1 \to {\mathbf{Met}}$ is an equivalence: indeed, if $(X,d)$ is a metric space and we define $d': X \times X \to [0,1]$ as
\[ d'(x,y) = {\rm min}(d(x,y), 1), \]
then $(X,d)$ and $(X,d')$ are isomorphic in ${\mathbf{Met}}_1$. So from a categorical point of view (and given our choice of morphisms in ${\mathbf{Met}}$), the restriction to the spaces having diameter at most 1 is not really a restriction. Of course, similar remarks apply to ${\mathbf{pMet}}$ vs.~${\mathbf{pMet}}_1$ and ${\mathbf{cMet}}$ vs.~${\mathbf{cMet}}_1$. 
\end{rema}

For the following elementary fact we were unable to find a reference, so we include a proof here.

\begin{prop}{pMet,Met has finite (co)limits}The categories $\mathbf{pMet}_1$, $\mathbf{Met}_1$ and $\mathbf{cMet}_1$ have finite limits.
    \end{prop}
    
    \begin{proof}
    Clearly, the terminal object in $\mathbf{pMet}_1$ as well as in $\mathbf{Met}_1$ and in $\mathbf{cMet}_1$ is the complete metric space $(X,d)$ where $X$ is a set with exactly one element and $d$ is the unique (pseudo)metric on $X$.
    
    Furthermore, the categories $\mathbf{pMet}_1$ and ${\mathbf{Met}}_1$ both have pullbacks: let $f: (X,d) \rightarrow (Z,d)$ and $g: (Y,d) \rightarrow (Z,d)$ be morphisms in $\mathbf{pMet}_1$ or in $\mathbf{Met}_1$. We have a pullback diagram:
    \begin{equation*}
    \begin{tikzcd}
    {(X \times_Z Y,d)} \arrow[d, "\pi_X"'] \arrow[r, "\pi_Y"] & {(Y,d)} \arrow[d, "g"] \\
    {(X,d)} \arrow[r, "f"']                                   & {(Z,d)}               
    \end{tikzcd}
    \end{equation*}
    Here, $(X \times_Z Y)$ is the usual pullback in ${\mathbf{Sets}}$ of $f$ and $g$. The morphisms $\pi_X$ and $\pi_Y$ are the projections in ${\mathbf{Sets}}$ and $d$ is the (pseudo)metric defined by $d((x,y),(x',y')) := \max (d_X(x,x'),d_Y(y,y'))$. That this gives a pullback is easy to check.
    
    Finally, in the case that $f$ and $g $ are morphisms in $\mathbf{cMet}_1$ we check that $((X \times_Z Y),d)$ is a complete metric space: suppose that $(x_n,y_n)_{n \in \N}$ is a Cauchy sequence in $X \times_Z Y$. Clearly $(x_n)_{n \in \mathbb{N}}$ and $(y_n)_{n \in \mathbb{N}}$ are Cauchy sequences. Denote their limits by $x$ and $y$, respectively. Let $\varepsilon \gt 0$; we wish to show that $d(f(x), g((y)) \leq \varepsilon$. Since $(x_n)_{n \in \N} \rightarrow x$ and $(y_n)_{n \in \N} \rightarrow y$ and $g$ and $f$ are uniformly continuous, there exists a $K \in \N$ such that for all $n \geq K$ it holds that $d(g(y_n),g(y)) \leq \frac{\varepsilon}{2}$ and $d(f(x_n),f(x)) \leq \frac{\varepsilon}{2}$. So if $n \geq K$ then by the triangle inequality
    \begin{align*}
        d(f(x),g(y)) &\leq d(f(x),f(x_n)) + d(f(x_n),g(y_n)) + d(g_n(y),g(y)) \\
        &\leq \varepsilon. 
    \end{align*}
    Since this holds for every $\varepsilon \gt 0$, we conclude that $d(f(x), g(y))= 0$ and $f(x) = g(y)$. Therefore $(x, y) \in X \times_Z Y$ and $(x_n,y_n)$ converges to $(x,y)$.
\end{proof}

In the sequel we will also use the following facts and definitions:

\begin{lemm}{cho2019categorical} {\rm [Lemma 3 in \cite{cho2019categorical}]}
    In $\mathbf{pMet}_1$ the monomorphisms are precisely the morphisms which are injective as functions.
    \end{lemm}

\begin{prop}{inf sup continuous} {\rm [Proposition 2.7 in \cite{yaacov_usvyatsov_2010}]}
    Suppose $M$, $M'$ are metric spaces and $f$ is a bounded uniformly continuous function from $M \times M'$ to $\mathbb{R}$. Then $\sup_{y} f$ and $\inf_y f$ are bounded uniformly continuous functions from $M$ to $\mathbb{R}$.
\end{prop}

\begin{theo}{Extreme value theorem} {\rm [Extreme value theorem (Theorem 4.16 in \cite{rudin_analysis})] }
Let $f: (X,d) \rightarrow \mathbb{R}$ be a continuous function from a metric space $(X,d)$. If $K \subseteq X$ is compact, then $f(K)$ is compact and there exist $p, q \in K$ such that $f(p) = \inf f(K)$ and $f(q) =  \sup f(K)$.
\end{theo}

\begin{defi}{Uniform spaces}
    A \textit{uniformity} for a set $X$ is a non-empty family $\mathcal{U}$ of subsets of $X \times X$ such that
    \begin{enumerate}
        \item each member of $\mathcal{U}$ contains the diagonal $\Delta$;
        \item if $U \in \mathcal{U}$, then $U^{-1} \in \mathcal{U}$;
        \item if $U \in \mathcal{U}$, then $V \circ V \subseteq U$ for some $V \in \mathcal{U}$;
        \item if $U$ and $V$ are members of $\mathcal{U}$, then $U \cap V \in \mathcal{U}$; \item if $U \in \mathcal{U}$ and $U \subseteq V \subseteq X \times X$, then $V \in \mathcal{U}$.
    \end{enumerate}
    Here $U^{-1} = \{ (y,x) \, : \, (x,y) \in U \}$ and $W \circ V = \{ (x,z) \, : \, (\exists y \in Y) \, (x,y) \in V \mbox{ and } (y,z) \in W \}$.
    The pair $(X,\mathcal{U})$ is a \textit{uniform space}.
\end{defi}

\begin{defi}{basisforuniformspace}
If $\mathcal{U}$ is a uniformity on a set $X$, then a collection $\mathcal{V} \subseteq \mathcal{U}$ is called a \emph{basis} for the uniformity $\mathcal{U}$ if for each $U \in \mathcal{U}$ there exists a $V \in \mathcal{V}$ with $V \subseteq U$.
\end{defi}

If $(X,d)$ is a metric space, write $V_\varepsilon = \{ (x, y) \in X \, : \, d(x,y) \leq \varepsilon \}$. Then 
\[ \mathcal{U} = \{ U \subseteq X \times X \, : \, (\exists \varepsilon \gt 0) V_\varepsilon \subseteq U \} \]
defines a uniformity on $X$. Such a uniformity has a countable basis given by $\{ V_q \, : \, q \in \mathbb{Q} \cap (0,1] \}$. It turns out that every uniformity with a countable basis can be obtained in this way.

\begin{theo}{Uniformity induced by metric}
A uniformity $\mathcal{U}$ on a set $X$ is induced by a metric on the set $X$ if and only if it has a countable basis.
\label{uniformity induced by metric}
\end{theo}

\begin{proof}
See Theorem 8.1.21 in \cite{Engelking}
\end{proof}

\section{Hyperdoctrines and their logic}

The purpose of the present section is to introduce the terminological and notational conventions that we will use when talking about hyperdoctrines and their internal logic. Those readers who need a proper introduction to the subject we refer to \cite{hyland_johnstone_pitts_1980}.

\subsection{Hyperdoctrines} Before we can define hyperdoctrines, we have to define indexed preorders.

\begin{defi}{indexed preorder}
Let $\mathcal{C}$ be a category with finite products. An \textit{indexed preorder} is a pseudofunctor $\mathbf{P}: \mathcal{C}^{op} \rightarrow {\mathbf{Preorders}}$. We call the category $\mathcal{C}$ the \textit{base} of $\mathbf{P}$. If $X$ is an object in $\mathcal{C}$ then we call elements of $\bfp(X)$ the \textit{predicates} on $X$.
\end{defi}

\begin{defi}{adjoints in a doctrine} Let $\mathbf{P}: \mathcal{C}^{op} \rightarrow {\mathbf{Preorders}}$ be an indexed preorder. We say $\mathbf{P}$ has \textit{left adjoints} if, for any projection $\pi: X \times Y \rightarrow X$ in $\mathcal{C}$, the morphism of preorders $\mathbf{P}(\pi)$ has a left adjoint $\exists_\pi^{\mathbf{P}}$ and we say $\mathbf{P}$ has \textit{right adjoints} if $\mathbf{P}(\pi)$ has a right adjoint $\forall_\pi^{\mathbf{P}}$.
Left (or right) adjoints satisfy the \textit{Beck-Chevalley condition} if
for any pullback diagram in $\mathcal{C}$
\begin{equation*}
\begin{tikzcd}
A \arrow[r, "\pi"] \arrow[d, "f"] & B \arrow[d, "g"] \\
C \arrow[r, "\pi'"]               & D   
\end{tikzcd}
\end{equation*}
where $\pi$ and $\pi'$ are projections, it holds that $\exists_{\pi}^\bfp \circ \mathbf{P}(f) \simeq \bfp(g) \circ \exists_{\pi'}^\bfp$ (or $\forall_{\pi}^\bfp \circ \mathbf{P}(f) \simeq \bfp(g) \circ \forall_{\pi'}^\bfp$, respectively). 
\end{defi}

\begin{rema}{onnotationofadj}
If $\mathbf{P}: {C}^{op} \rightarrow {\mathbf{Preorders}}$ is an indexed preorder, and $f$ is a morphism in $\mathcal{C}$, we will often write $f^*$ instead of $\mathbf{P}(f)$ and similarly $\exists_f$ instead of $\exists_{f}^{\mathbf{P}}$ and $\forall_f$ instead of $\forall_f^{\mathbf{P}}$.
\end{rema}

\begin{defi}{frobenius}
We say that an indexed preorder $\bfp$ satisfies the \textit{Frobenius condition}{{{}}} if  $\bfp$ has left adjoints and for any projection $\pi: B \rightarrow A$ and any $\alpha \in \bfp(A)$ and $\beta \in \bfp(B)$ it holds that \[ \exists_\pi(\bfp(\pi)(\alpha) \wedge \beta) \simeq \alpha \wedge \exists_\pi(\beta). \]  
\end{defi}

\begin{defi}{equality in a doctrine}
We say that an indexed preorder $\bfp$ \textit{has equality} if the following two conditions hold:
\begin{enumerate}
    \item For any diagonal $\delta: X \rightarrow X \times X$ in $\mathcal{C}$, there exists an element $Eq_X$ (the \textit{equality predicate}) such that for any $A \in \bfp(X \times X)$ the following holds (assuming that $\bfp(X)$ has a top element $\top_X$): $$ \top_X \leq \bfp(\delta)(A) \iff Eq_X \leq A$$ 
    \item If $\pi_{13}: A \times B \times A \times B \rightarrow A \times A$ and $\pi_{24}: A \times B \times A \times B \rightarrow B \times B$ are the obvious projections in $\mathcal{C}$, then 
    $$\bfp(\pi_{13})(Eq_{A}) \wedge \bfp(\pi_{24})(Eq_{B}) \simeq Eq_{A \times B}$$
\end{enumerate}
\end{defi}

\begin{defi}{Hyperdoctrines}
    Let $\mathcal{C}$ be a category with finite products.
    \begin{itemize}
        \item A \textit{coherent hyperdoctrine} is an indexed preorder $\mathcal{C}^{op} \rightarrow {\mathbf{DistPrelattices}}$  with equality, that satisfies the Frobenius condition, and in which the left adjoints satisfy the Beck-Chevalley condition.
        \item A \textit{universally coherent hyperdoctrine} is a coherent hyperdoctrine that has right adjoints that satisfy the Beck-Chevalley condition.
    \end{itemize}
\end{defi}

\begin{defi}{morphisms of hyperdoctrines}
        Let $\mathbf{P}$ with base $\mathcal{C}$ and $\mathbf{P}'$ with base $\mathcal{D}$ be (universally) coherent hyperdoctrines. A \textit{morphism of (universally) coherent hyperdoctrines} $\mathbf{P} \rightarrow \mathbf{P}'$ is a pair $(F,\eta)$ where $F: \mathcal{C} \rightarrow \mathcal{D}$ is a functor preserving finite products and $\eta: \bfp \rightarrow \bfp' \circ F^{op}$ is a pseudonatural transformation satisfying the following properties:
        \begin{enumerate}
            \item for each projection $\pi: Y \rightarrow X$ in $\mathcal{C}$ it holds that $\eta_X \circ \exists^\bfp_\pi \simeq \exists_\pi^{\bfp' \circ F^{op}} \circ \eta_Y$ (and $\eta_X \circ \forall^\bfp_\pi \simeq \forall_\pi^{\bfp' \circ F^{op}} \circ \eta_Y$ in the case of universally coherent hyperdoctrines)
            \item for every object $X$ in $\mathcal{C}$ it holds that $\eta_{X \times X} (Eq_{X}) \simeq \bfp'(\ang{F{\pi_1},F\pi_2})(Eq_{FX})$, where $\pi_1, \pi_2$ are the projections $X \times X \rightarrow X$.
        \end{enumerate}
If $(\id_{\mathcal{C}},\eta)$ is a morphism between hyperdoctrines we also denote it by $\eta$.
    \end{defi}

\begin{defi}{embedding}
Let $(F,\eta): \mathbf{P} \rightarrow \mathbf{Q}$ be a morphism of hyperdoctrines{{{}}} where $\mathbf{P}$ has base $\mathcal{C}$.  We say $(F,\eta)$ is \textit{full}, or \textit{an embedding}, if $F$ is full and faithful and for each object $X$ in $\mathcal{C}$ and for each pair $A, B \in \mathbf{P}(X)$ we have 
$$\eta_X(A) \leq \eta_X(B) \text{ if and only if } A \leq B$$
\end{defi}

The following ``change of base'' property of hyperdoctrines will be used later. The straightforward proof is omitted.

\begin{prop}{change of base: qhd}
Let $\bfp: \mathcal{C}^{op} \rightarrow {\mathbf{DistPrelattices}}$ be a coherent hyperdoctrine and let $\mathcal{D}$ be a category with finite products. Let $F: \mathcal{D} \rightarrow \mathcal{C}$ be a functor that preserves finite products. Then $\bfp \circ F^{op}$ is a coherent hyperdoctrine and $(F, \id)$ a morphism of coherent hyperdoctrines. Furthermore, if $\mathbf{P}$ is universally coherent, then $P \circ F^{op}$ is universally coherent and $(F,\id)$ is a morphism of universally coherent hyperdoctrines.
\end{prop}

\subsection{Interpretation of first-order logic in a hyperdoctrine} \label{Subsection: internal language of a hyperdoctrine}

The logical symbols we will consider are those of coherent logic with equality ($=$, $\top$, $\bot$, $\vee$, $\wedge$, $\exists$) and universal quantification ($\forall$). 

We work in a multi-sorted logic. The class of \textit{coherent formulas} is built inductively using the symbols of coherent logic with equality ($=$, $\top$, $\bot$, $\vee$, $\wedge$, $\exists$). The class of \textit{universally coherent formulas} is built similarly, but with the addition of the universal quantification symbol ($\forall$).

A \textit{context} $\bar x$ is a list of sorted variables $x_1^{S_1},...,x_n^{S_n}$. If $t$ is a term whose free variables are contained in $\bar x$, then we write $(t | \bar x)$ to denote the term $t$ \textit{in context} $\bar x$.  Likewise, for a formula $\phi$ whose free variables are contained in context $\bar x$ we denote the formula $\phi$ \textit{in context} $\bar x$ by $(\phi | \bar x)$. 

A \textit{sequent} over a signature $\Sigma$ is an expression of the form $\phi \vdash_{\bar x} \psi$ or an expression of the form $\phi \dashv \hspace{1pt} \vdash_{\bar x} \psi$ where $\bar x$ is a context and $\phi$ and $\psi$ are formulas over $\Sigma$ whose free variables are contained in $\bar x$. We say a sequent $\phi \vdash_{\bar x} \psi$ is (universally) coherent if $\psi$ and $\phi$ are (universally) coherent formulas. 

\begin{defn}\label{defn: interpretation of FOL in HD}
For an indexed preorder $\mathbf{P}: \mathcal{C}^{op} \rightarrow {\mathbf{Preorders}}$ and a signature $\Sigma$, an interpretation $\brac{\cdot}$ of $\Sigma$ in $\mathbf{P}$ is an assignment as follows:
\begin{enumerate}
    \item for each sort $S$ in $\Sigma$ an object $\brac{S}$ in $\mathcal{C}$;
    \item for each function symbol $f: S_1 \times ... \times S_n \rightarrow S$ in $\Sigma$ a morphism $\brac{f}: \brac{S_1} \times ... \times \brac{S_n} \rightarrow \brac{S}$ in $\mathcal{C}$;
    \item for each relation symbol $R \subseteq (S_1, S_2, ... , S_n)$ in $\Sigma$ an element $\brac{R}$ of $\mathbf{P}(\brac{S_1} \times ... \times \brac{S_n})$.
\end{enumerate}
\end{defn}

If $\bar x = x_1^{S_1}, ..., x_n^{S_n}$ is a context, then we use the notation $\mathbf{s}(\bar x) :=  \brac{S_1} \times ... \times \brac{S_n}$ and if $t$ is a term of sort $S$ we say $\mathbf{s}(t) = \brac{S}$. 
If $\bfp$ is a coherent hyperdoctrine (resp. a universally coherent hyperdoctrine), the interpretation of terms-in-context and formulas-in-context goes according to clauses \ref{hdclause basic}-\ref{hdclause coh} (respectively, according to clauses \ref{hdclause basic}-\ref{hdclause quant})
\begin{multicols}{2}
\begin{enumerate}[label=(\roman*)]
    \item $\brac{x_i^{S_i}}_{\bar x} = \pi_i$ \label{hdclause basic}
    \item $\brac{f(t_1,...,t_n)}_{\bar x} = \\ \brac{f} \circ \ang{\brac{t_1}_{\bar x}, ... ,\brac{t_n}_{\bar x}}$ 
    \item $\brac{R(t_1,...,t_n)}_{\bar x} = \\ \ang{\brac{t_1}_{\bar x}, ... , \brac{t_n}_{\bar x}}^*(\brac{R})$
    \item $\brac{s={}t} = \ang{\brac{s}_{\bar{x}}, \brac{t}_{\bar x}}^* (Eq_{\mathbf{s}(\bar x)})$
    \item $\brac{\top}_{\bar x} = \top$
    \item $\brac{\bot}_{\bar x} = \bot$
    \item $\brac{\phi \wedge \psi}_{\bar x} = \brac{\phi}_{\bar x} \wedge \brac{\psi}_{\bar x}$
    \item $\brac{\phi \vee \psi}_{\bar x} = \brac{\phi}_{\bar x} \vee \brac{\psi}_{\bar x}$
    \item $\brac{ \exists y( \phi)}_{\bar x} = \exists_{\pi}(\brac{\phi}_{\bar x y})$ \label{hdclause coh}
    \item $\brac{ \forall y( \phi)}_{\bar x} = \forall_{\pi}(\brac{\phi}_{\bar x y})$ \label{hdclause quant}
\end{enumerate}
\end{multicols}

where $\pi_i: \mathbf{s}(\bar x) \rightarrow \brac{S_i}$ and $\pi: \mathbf{s}(\bar x) \times \mathbf{s}(y) \rightarrow \mathbf{s}(\bar x)$ are projections.

Given an interpretation $\brac{\cdot}$ of a signature $\Sigma$ in a hyperdoctrine $\mathbf{P}$ and given a sequent $\phi \vdash_{\bar x} \psi$ over $\Sigma$, we say that the sequent \textit{holds} or is \textit{valid} with respect to $\brac{\cdot}$ if $ \brac{\phi}_{\bar x} \leq \brac{\psi}_{\bar x}$ in $\mathbf{P}(\mathbf{s}(\bar x))$. We say that $\phi \dashv \hspace{1pt} \vdash_{\bar x} \psi$ holds if $ \brac{\phi}_{\bar x} \leq \brac{\psi}_{\bar x}$ and $ \brac{\phi}_{\bar x} \geq \brac{\psi}_{\bar x}$. When a context $\bar{x}$ for such sequents is implicitly clear, we may simply write that $\phi \vdash \psi$ holds or that $\phi {\ \dashv \hspace{1pt} \vdash \ } \psi$ holds. We say that a formula-in-context $(\phi|\bar x)$ is \textit{true} if the sequent $\top \vdash_{\bar x} \phi$ holds.

\begin{definition}
    Suppose $\bfp: \mathcal{C}^{op} \rightarrow {\mathbf{Preorders}}$ and $\bfq: \mathcal{D} \rightarrow {\mathbf{Preorders}}$ are indexed preorders and that $F: \mathcal{C} \rightarrow \mathcal{D}$ is a functor that preserves finite products. Let $\eta: \bfp \rightarrow \bfq$ be a pseudonatural transformation. Let $\Sigma$ be a signature and let $\brac{\cdot}$ be an interpretation of $\Sigma$ in $\bfp$. We define the interpretation $\brac{\cdot}'$ of $\Sigma$ in $\bfq$ as follows:
    
    \begin{enumerate}
        \item $\brac{S}' = F( \brac{S})$ for all sorts $S \in \Sigma$
        \item $\brac{f}' = F(\brac{f}) \circ \ang{F\pi_1,...,F\pi_n}^{-1}$ for all functions symbols $f: S_1 \times  ... \times S_n \rightarrow S \in \Sigma$
        \item $\brac{R}' = \bfq(\ang{F\pi_1,...,F\pi_n})(\eta_{(\brac{S_1} \times ... \times \brac{S_n})}( \brac{R} ))$ for all relation symbols $R \in \Sigma$ of sort $S_1 \times ... \times S_n$
    \end{enumerate}
    Here $\pi_i$ denotes the projection $\brac{S_1} \times ... \times \brac{S_n} \rightarrow \brac{S_i}$ for any $i \leq n$ and we note that $\ang{F\pi_1,...,F\pi_n}^{-1}$ exists since $F$ preserves finite products. 
    We denote the interpretation $\brac{\cdot}'$ by $(F,\eta)(\brac{\cdot})$.
\end{definition}

The following proposition, which can be proven via induction on formulas, states that (full) morphisms of hyperdoctrines preserve (and reflect) the validity of sequents.

\begin{prop}{morphism of models}
    Let $(F,\eta): \bfp \rightarrow \bfq$ be a morphism of (universally) coherent hyperdoctrines, and let $\brac{\cdot}$ be an interpretation of a signature $\Sigma$ in $\bfp$. Then for any (universally) coherent sequent $S$ over $\Sigma$ the following holds:
    \begin{align*}
         S \text{ is valid in } \brac{\cdot} \implies S \text{ is valid in } (F,\eta)(\brac{\cdot})
    \end{align*}
    If moreover $(F,\eta)$ is full, then the following holds:
    \begin{align*}
        S \text{ is valid in } \brac{\cdot} \iff S \text{ is valid in } (F,\eta)(\brac{\cdot})
    \end{align*}
\end{prop}

\section{A hyperdoctrine for continuous model theory}

In this section we will introduce a hyperdoctrine for continuous model theory. In continuous model theory the sorts are interpreted as complete metric spaces, say of diameter 1, while the function symbols are interpreted as uniformly continuous functions. So this suggests that the base of the hyperdoctrine $\mathbf{CMT}$ for continuous model theory should be the category $\mathbf{cMet}_1$. The literature on continuous model theory also leaves no doubt about what predicates should be: the predicates on an element $(X,d)$ in $\mathbf{cMet}_1$ are the uniformly continuous functions $\alpha: (X, d) \to [0,1]$, while for $f: (Y,d) \rightarrow (X,d)$ in $\mathbf{cMet}_1$ we define $\mathbf{CMT}_{}(f): \mathbf{CMT}_{}(X,d) \rightarrow \mathbf{CMT}_{}(Y,d)$ as the function taking $\alpha$ to $\alpha \circ f$. 

That this is how things should be can be glanced from any text on continuous model theory. But what is not at all clear is how the predicates should be \emph{(pre)ordered}. Indeed, this has to do with the absense of any notion of logical consequence in continuous model theory. In \cite{benyaacovetal08} we find:
\begin{quote}
One of the subtleties of continuous first-order logic is that it is essentially a positive logic. In particular, there is no direct way to express an implication between conditions. This is inconvenient in applications, since many natural mathematical properties are stated using implications.
\end{quote}
But the authors of \cite{benyaacovetal08} go on to say that:
\begin{quote}
    However, when working in a saturated model or in all models of a theory, this obstacle can be clarified and often overcome in a natural way, which we explain here.
    
    Let $L$ be any signature for metric structures. For the rest of this section we fix two $L$-formulas $\varphi(x_1,...,x_n)$ and $\psi(x_1,...,x_n)$ and an $L$-theory $T$. For convenience we write $x$ for $x_1, \ldots , x_n$.

\noindent    
{\bf 7.14 Proposition.} Let $\mathcal{M}$ be an $\omega$-saturated model of $T$. The following statements are equivalent:
\begin{enumerate}
\item[(1)] For all $a \in \mathcal{M}^n$, if $\varphi_{\mathcal{M}}(a) = 0$ then $\psi_{\mathcal{M}}(a) = 0$. 
\item[(2)] $\forall \varepsilon > 0 \, \exists \delta > 0 \, \forall a \in \mathcal{M}^n ( \, \varphi_{\mathcal{M}}(a)<\delta \Rightarrow  \psi_{\mathcal{M}}(a) \leq \varepsilon \, )$.
\item[(3)] There is an increasing, continuous function $\alpha: [0,1] \to [0,1]$ with $
\alpha(0)=0$ such that $\psi (a) \leq \alpha(\varphi (a))$ for all $a \in \mathcal{M}^n$.
\end{enumerate}
\end{quote}
This passage, and in particular point (2) in the proposition above, is the closest to a suggestion for a notion of consequence for continuous model theory that we are aware of. What we will do in this paper is take this up and see where this gets us. That is, we will introduce the following preorder on the predicates of $\mathbf{CMT}$: for any two uniformly continuous functions $\alpha, \beta: (X,d) \to [0,1]$ we will write $\alpha \sq \beta$ if for any $\varepsilon >0$ there exists a $\delta>0$ such that for all $x \in X$ we have that $\alpha(x) < \delta$ implies $\beta(x) \leq \varepsilon$. 

In this section we will show that this yields a universally coherent hyperdoctrine, where the interpretation of the connectives mirrors the standard interpretation of logic in continuous model theory.

Besides having its basis in the literature on continuous model theory and the fact that it gives the correct interpretation of logical formulas, there are two more reasons to think that this is the correct way of ordering the predicates. First of all, given the prominent place of the concept of uniform continuity in continuous model theory, it makes sense that the ordering of the predicates reflects this fact. Secondly, there is a stripped down version of the hyperdoctrine $\mathbf{CMT}$, which we call $\mathbf{U}$. In the next section we will prove that the category of complete metric spaces and uniformly continuous maps can be seen as the category of equivalence relation over $\mathbf{U}$. This gives a logical reconstruction of the category ${\mathbf{cMet}}_1$ purely in terms of this way of ordering predicates. 

All of this suggests to us that for the purposes of continuous model theory this choice of preordering the predicates is the right one. There is also a problem, however, which we will discuss after we have proved that $\mathbf{CMT}$ is a universally coherent hyperdoctrine. 

\subsection{The hyperdoctrine $\mathbf{U}$} Before we discuss $\mathbf{CMT}$, we will first discuss the simplified version of it which we have called $\mathbf{U}$; this hyperdoctrine $\mathbf{U}$ will also play a crucial role in the embedding of $\mathbf{CMT}$ into the subobject hyperdoctrine of a topos.

The starting point for the hyperdoctrine $\mathbf{U}$ is that the ordering $\alpha \sq \beta$ makes sense more generally.

\begin{defi}{ordering sq definition} Let $X$ be any set. Let $\alpha,\beta: X \rightarrow [0,1]$. We write $\alpha \sq \beta$ if for any $\varepsilon >0$ there exists a $\delta>0$ such that for all $x \in X$ we have that $\alpha(x) \leq \delta$ implies $\beta(x) \leq \varepsilon$.
\end{defi}

The same ordering can be defined in other ways as well:
\begin{prop}{ordering modulus of continuity}\rm{(Proposition 2.10 in \cite{yaacov_usvyatsov_2010})} 
Let $\alpha, \beta: X \rightarrow [0,1]$ be arbitrary functions. 
The following conditions are equivalent:
\begin{enumerate}
    \item $\alpha \sq \beta$. 
    \item There exists an increasing function $\Delta: (0,1] \rightarrow (0,1]$ such that for any $\varepsilon > 0$ and $x \in X$ we have that $\alpha(x) \leq \Delta(\varepsilon)$ implies $\beta(x) \leq \varepsilon$.
    \item There exists an increasing, continuous function $F: [0,1] \rightarrow [0,1]$ such that $F(0) = 0$ and for all $x \in X$ we have that $\beta(x) \leq F(\alpha(x))$.
\end{enumerate}
\end{prop}

\begin{rema}{modulusofunifcont}
If $\alpha,\beta$ is a pair of functions with $\alpha \sq \beta$ then we call the function $\Delta$ in the second condition of \refprop{ordering modulus of continuity} a \textit{modulus of uniform continuity} for $\alpha \sq \beta$. Obviously, $\Delta$ is not unique. The ordering in \refdefi{ordering sq definition} can be used to characterise uniform continuity of functions between metric spaces: if $(X,d_X)$ and $(Y,d_Y)$ are objects in $\mathbf{cMet}_1$ and $f$ is a function $X \rightarrow Y$ then $f$ is a uniformly continuous function $(X,d_X) \rightarrow (Y,d_Y)$ if and only if $d_X \sq (d_Y \circ (f \times f))$. In this case, our definition of a modulus of uniform continuity is the familiar notion of a modulus of uniform continuity from analysis.
\end{rema}

In addition, we have the following intuitive representation of the statement $\alpha \sq \beta$ that is based on the convergence of sequences to zero. We omit the simple proof.
\begin{prop}{Preliminaries, characterisation of sq}
Let $X$ be any set. Let $\alpha,\beta: X \rightarrow [0,1]$. Then the following conditions are equivalent:
\begin{enumerate}[label=\arabic*)]
    \item $\alpha \sq \beta$.
    \item For any sequence $(x_n)_{n \in \N}$ in $X$ we have that $\alpha(x_n) \rightarrow 0$ implies $\beta(x_n) \rightarrow 0$.
\end{enumerate}
\end{prop}

\begin{rema}{convergencetothetruth}
Since we think of 0 as the truth in continuous logic, one can put the second condition in the previous proposition in words as follows: whenever $\alpha$ convergences to the truth, then so does $\beta$. 
\end{rema}

If $X$ is a set, then $\sq$ is a preordering on the set of all functions $X \rightarrow [0,1]$. In fact, this gives us a hyperdoctrine, as we will see in \refprop{Sub_UC definition}. 

\begin{defi}{P_UC}
    If $X$ is any set, then we define $\mathbf{U}(X)$ to be be the preorder $(\{X \rightarrow [0,1]\}, \sq)$. If $f: X \rightarrow Y$ is a function, then we let $\mathbf{U}(f): \mathbf{U}(Y) \rightarrow \mathbf{U}(X)$ be the function that maps an element $\alpha \in \mathbf{U}(Y)$ to $\alpha \circ f$.
\end{defi}

\begin{prop}{Sub_UC definition}
    The assignment $X \mapsto \mathbf{U}(X)$ for any set $X$ and $f \mapsto \mathbf{U}(f)$ for any function $f$ determines a universally coherent hyperdoctrine $\mathbf{U}: {\mathbf{Sets}}^{op} \rightarrow {\mathbf{DistPrelattices}}$.
\end{prop}

\begin{proof}
We first check that $\mathbf{U}$ is a functor $\mathbf{U}: {\mathbf{Sets}}^{op} \rightarrow {\mathbf{DistPrelattices}}$.
The top element of $\mathbf{U}(X)$ is given by the function that is zero everywhere. The bottom element of $\mathbf{U}(X)$ is any function $f: X \rightarrow [0,1]$ such that $\inf_{x \in X} f(x) > 0$, such as the constant 1 function. Binary meets and binary joins are the maximum and minimum, respectively. That $\mathbf{U}(X)$ is distributive follows immediately from the fact that the minimum distributes over the maximum.

If $f: Y \rightarrow X$ is a function, then $\mathbf{U}f$ is clearly order preserving. In addition, it preserves both top and bottom elements. For finite joins, note that $f^*(\alpha \wedge \beta)(y) = \min( \alpha(f(y)), \beta(f(y)) )= (f^*(\alpha) \wedge f^*(\beta))(y)$ for any $\alpha, \beta \in \mathbf{U}(X)$ and $y \in Y$. The case of meets is  similar. In addition, functoriality of $\mathbf{U}$ follows directly from the fact that function composition is associative and unital.

Next, we construct left adjoints to pullback: Given functions $f: Y \rightarrow X$ and $\alpha: Y \rightarrow [0,1]$ we define for $x \in X$
\begin{equation*}
    (\exists_f\alpha)(x) := 
    \ \inf \{ \alpha(y) \ | \ y \in Y, f(y) = x \}
\end{equation*}
We show that $\exists_f$ is order-preserving $\mathbf{U}(X) \rightarrow \mathbf{U}(Y)$. Suppose we are given $\alpha, \beta \in \mathbf{U}(Y)$ with $\alpha \sq \beta$  and modulus of uniform continuity $\Delta$. Let $\varepsilon \in (0,1)$. Note that the infimum of the empty set is equal to $1$. If $(\exists_f \alpha)(x) \leq \frac{1}{2} \Delta(\varepsilon)$, then $x = f(y)$ for some $y \in Y$ with $\alpha(y) \leq \Delta (\varepsilon)$. But then $\beta(y) \leq \varepsilon$ so in particular $\inf \{\beta(y) \ | \ y \in Y, f(y) = x\} \leq \varepsilon$. This gives $\exists_f\beta(x) \leq \varepsilon$; hence $\exists_f(\alpha) \sq \exists_{f}(\beta)$ and $\exists_f$ is order-preserving.

Let $f: Y \rightarrow X$. We want to show that $\exists_f \dashv \mathbf{U}(f)$, or in other words, that for any $\alpha \in \mathbf{U}(Y)$ and $\beta \in \mathbf{U}(X)$ it holds that $\alpha \sq \beta \circ f$ if and only if $\exists_f \alpha \sq \beta$. So suppose that $\alpha \sq \beta \circ f$ with modulus of uniform continuity $\Delta_1$. Let $\varepsilon \in (0,1)$. If $(\exists_f \alpha)(x) \leq \frac{1}{2} \Delta_1 (\varepsilon)$, then there exists a $y \in Y$ such that $f(y) = x$ and $\alpha(y) \leq \Delta_1(\varepsilon)$ so $(\beta \circ f) (y) \leq \varepsilon$. But then $\beta(x) = \beta(f(y)) \leq \varepsilon$ hence $\exists_f \alpha \sq \beta$. 

Conversely, suppose $\exists_f \alpha \sq \beta$ with corresponding modulus $\Delta_2$. Let $\varepsilon \in (0,1)$. If $\alpha(y) \leq \Delta_2(\varepsilon)$ then $(\exists_f \alpha)(f(y)) = \inf\{\alpha(y) \ | \ f(y) = f(y) \} = \alpha(y) \leq \Delta_2(\varepsilon)$ which implies that $\beta(f(y)) \leq \varepsilon$ hence $\alpha \sq \beta \circ f$.

We also have right adjoints to pullback: for any function $f: Y \rightarrow X$ and function $\alpha: Y \rightarrow [0,1]$ and $x \in X$ we let
\begin{equation*}
    (\forall_f\alpha)(x) := 
    \sup \{ \alpha(y) \ | \ y \in Y, f(y) = x \}
\end{equation*}
Let $f: Y \rightarrow X$. To show that $\forall_f$ is order-preserving, suppose that $\alpha \sq \gamma$ in $\mathbf{U}(Y)$ with corresponding modulus of uniform continuity $\Delta$. Let $\varepsilon \in (0,1)$. Suppose $\forall_f \alpha (x) \leq \Delta(\varepsilon)$ for some $x \in X$. If $x = f(y)$ for some $y \in Y$ then $\alpha(y) \leq \sup \{ \alpha(y') \ | \ y' \in Y, f(y') = x \} \leq \Delta(\varepsilon)$ so $\gamma(y) \leq \varepsilon$. Then it follows that $\forall_f \gamma(x) = \sup \{ \gamma (y') | y \in Y, f(y') = x\} \leq \varepsilon$. Therefore $\forall_f(\alpha) \sq \forall_f(\gamma)$, hence $\forall_f$ is order-preserving.

 Next, let $f: Y \rightarrow X$. We show that $\mathbf{U}(f) \dashv \forall_f$. Suppose that $\alpha \in \mathbf{U}(Y)$ and $\beta \in \mathbf{U}(X)$ such that $f^* \beta \sq \alpha$ with modulus of uniform continuity $\Delta$. Suppose that $x \in X$ such that $\beta(x) \leq \Delta(\varepsilon)$. If $f(y) = x$ for some $y \in Y$, then $f^*\beta (y) = \beta(f(y)) = \beta (x) \leq \Delta(\varepsilon)$, so $\alpha(y) \leq \varepsilon$. Therefore $\forall_f \alpha(x) \leq \varepsilon$ and we deduce that $\beta \sq \forall_f \alpha$.

Conversely, suppose that $\beta \sq \forall_f(\alpha)$ with modulus $\Delta$. Let $y \in Y$ with $\beta(f(y)) \leq \Delta(\varepsilon)$. Then $\alpha(y) \leq \sup \{\alpha(y') \ | \ y' \in Y, f(y') = f(y) \} = \forall_f(\alpha)(f(y)) \leq \varepsilon$, hence $f^*\beta \sq \alpha$.

Next, we show that $\mathbf{U}$ satisfies the Beck-Chevalley condition: Let 
    \begin{equation}
    \begin{tikzcd}
    D \arrow[r, "g"] \arrow[d, "f"'] & C \arrow[d, "h"] \\
    B \arrow[r, "k"']                & A             
    \end{tikzcd} \label{Beck-Chevalley 1}  
    \end{equation}
be a pullback in ${\mathbf{Sets}}$. We need to show that $\exists_g \circ f^* \simeq h^* \circ \exists_k$. Since (\ref{Beck-Chevalley 1}) is a pullback, it is equivalent to the following diagram in ${\mathbf{Sets}}$.
    \begin{equation}
    \begin{tikzcd}
    D' \arrow[r, "\pi_C"] \arrow[d, "\pi_B"'] & C \arrow[d, "h"] \\
    B \arrow[r, "k"']                & A             
    \end{tikzcd}
    \end{equation}
    Here $D' := \{ (b,c) \in B \times C \ | \ k(b) = h(c) \}$ and $\pi_B$ and $\pi_C$ are the projections. More precisely, there exists an isomorphism $\eta: D' \rightarrow D$ such that $f \circ \eta = \pi_B$ and $g \circ \eta = \pi_C$. \\
    Let $c \in C$ and let $\beta \in \mathbf{U}(B)$. We can write the following:
    \begin{align}
        (h^* \circ \exists_k)(\beta)(c) &= \sup \{\beta(b) \ | \ b \in B, k(b) = h(c) \} \\
        &= \sup \{ \beta ( \pi_B ( (b,c') ) \ | \ (b,c') \in D', \pi_C ((b,c')) = c \} \\
        &= \sup \{ \beta ( f ( \eta ( (b,c') ))) \ | \ (b,c') \in D', g( \eta ((b,c'))) = c \} \\
        &= \sup \{ \beta ( f(d)) \ | \ d \in D , g(d) = c\} \label{Beck-Chevalley 2}\\
        &= (\exists_g \circ f^*) (\beta)(c)
    \end{align}
    Here we use that $\eta$ is an isomorphism with $f \circ \eta = \pi_B$ and $g \circ \eta = \pi_C$ to obtain (\ref{Beck-Chevalley 2}). 
    
Lastly, we show that $\mathbf{U}$ satisfies the Frobenius condition: Let $i: Y \rightarrow X$ be a function and $\phi \in \mathbf{U}(X)$ and let $\psi \in \mathbf{U}(Y)$ and $x \in X$. We get the following:
\begin{align*}
    ((\exists_i \psi) \wedge \phi)(x) &= \max(\phi(x), \inf \{ \psi(y) \ | \ y \in Y, i(y) = x \}) \\
    &= \inf \{ \max(\psi(y), \phi(x))  \ | \ y \in Y, i(y) = x \} \\
    &= \inf \{ \max(\psi(y),\phi ( i(y))) \ | \ y \in Y , i(y) = x\} \\
    &= \exists_i( \psi \wedge i^*(\phi))(x)
\end{align*}
\end{proof}

The hyperdoctrine $\mathbf{U}$ is not first-order, as it lacks a Heyting implication, in general.

\begin{prop}{P(X) no Heyting implication}
If $X$ is an infinite set then $\mathbf{U}(X)$ does not have a Heyting negation; for that reason, it also does not have a Heyting implication.
\end{prop}
\begin{proof}
This is equivalent to showing that there exists a function $\beta \in \mathbf{U}(X)$ such that for all functions $\alpha \in \mathbf{U}(X)$ there exists a function $\gamma \in \mathbf{U}(X)$ such that the following does not hold:
\begin{align*}
     \gamma \wedge \beta \sq \bot \iff \gamma \sq \alpha
\end{align*}

Since $X$ is infinite, there exists a sequence $(x_n)_{n \in \N}$ in $X$ such that $x_i \neq x_j$ for any $i \neq j$. Define 
\begin{align*}
    \beta(x) = 
    \begin{cases}
    \frac{1}{n} & \text{ if } x = x_n \\
    1 & \text{ otherwise }
    \end{cases}
\end{align*}
Clearly $\beta \not \sq \bot$. If $\alpha = \top$ then for $\gamma:= \beta$ we get $\gamma \wedge \beta = \beta \not \sq \bot$ and $\gamma \sq \top$. If $\alpha \neq \top$, then $\alpha(y) > 0$ for some $y \in X$. Define $\gamma \in \mathbf{U}(X)$ such that $\gamma(y) = 0$ and such that $\gamma(x) = 1$ for all $x \not = y$. It follows that $\gamma \wedge \beta \simeq \bot$, since $\max(\gamma(x), \beta(x)) \geq \beta(y) > 0$ for all $x \in X$. But $\gamma \not \sq \alpha$ because $\alpha(y) > 0$ and $\gamma(y) = 0$.
\end{proof}

The previous result tells us that while we have managed to preorder the predicates in $\mathbf{U}$, we cannot in general ``internalise'' that relation.

\subsection{The hyperdoctrine $\mathbf{CMT}$} Recall that $\mathbf{CMT}$ is the indexed preorder with base $\mathbf{cMet}_1$, and for which ${\mathbf{CMT}}(X,d)$ is the set of uniformly continuous functions $(X,d) \to [0,1]$, with the operation of pullback given by precomposition.

\begin{prop}{CMThyperdoctrine} The indexed preorder ${\mathbf{CMT}}$ is a universally coherent hyperdoctrine. In addition, if $F$ is the forgetful functor ${\mathbf{cMet}}_1 \to {\mathbf{Sets}}$ then $(F, \id)$ is an embedding of universally coherent hyperdoctrines ${\mathbf{CMT}} \to \mathbf{U}$.
\end{prop}    
\begin{proof}
    This follows from (the proof of) \refprop{Sub_UC definition}. We only have to verify that the operations, like maximum, minimum, infimum and supremum as in $\mathbf{U}$, restrict to those of $\mathbf{CMT}$. But this follows from \refprop{inf sup continuous}.
\end{proof}

\begin{prop}{anotherembedding} If $D: {\mathbf{Sets}} \to {\mathbf{cMet}}_1$ is the functor sending a set to the discrete metric on that set, then $(D, \id)$ is an embedding of universally coherent hyperdoctrines $\mathbf{U} \to {\mathbf{CMT}}$. In fact, it is bijective on the predicates. Hence also ${\mathbf{CMT}}$ does not have a Heyting implication or negation.
\end{prop}
\begin{proof}
    This follows from the fact that if we give a set $X$ the discrete metric, then any function $\alpha: X \to [0,1]$ becomes uniformly continuous.
\end{proof}

\begin{rema}{oncompactness} Our counterexample showing that ${\mathbf{CMT}}(X)$ does not have a Heyting negation crucially relies on the fact that we may choose $X$ to be a non-compact space. An interesting question is whether there are examples of compact spaces $X$ for which ${\mathbf{CMT}}(X)$ does not have a Heyting negation or implication.
\end{rema}

\subsection{Discussion} In what follows the hyperdoctrine $\mathbf{CMT}$ will be our starting point for our categorical analysis of the logic of continuous model theory. Our main innovation is the introduction of a preorder on the predicates, which is necessary to obtain hyperdoctrine in the sense of Lawvere and make the embedding into a sheaf topos possible. However, some aspects of continuous model theory are hard to square with this way of preordering the predicates, as we will now discuss.

For us, the class of formulas which can be interpreted in $\mathbf{CMT}$ are those built from the atomic formulas using the quantifiers $( \inf, \sup)$ and the connectives $(\max, \min)$. However, the literature on continuous model theory suggestes that any uniformly continuous operation $u: [0,1]^n \rightarrow [0,1]$ should be considered as a logical connective. This is problematic for us, because such operations $u$ do not preserve logical equivalence (or isomorphism) in the sense of ${\mathbf{CMT}}$: in that sense they are too intensional. Those operations that preserve logical equivalence in the sense of $\mathbf{CMT}$ are characterised in the following proposition.

\begin{prop}{bad connectives for CMT}
    Let $n \geq 1$ and let $u: ([0,1]^n,d_n) \rightarrow ([0,1],d) $  be uniformly continuous, where $([0,1]^n, d_n)$ is the $n$-ary product of the metric space $([0,1],d)$, with $d$ the standard metric.  The following statements are equivalent:
\begin{enumerate}
    \item For any complete metric space $(Y,d)$ and for any uniformly continuous $\phi_i$ and $\phi_i'$ where $i \in \{1,...,n\}$ and which are elements of ${\mathbf{CMT}}(Y)$ such that $\phi_i \simeq \phi_i'$, it holds that $u(\phi_1,...,\phi_n) \simeq u(\phi'_1,...,\phi_n')$.
    \item For any $p, q \in [0,1]^n$ such that $p_i = 0$ if and only if $q_i =0$ one of the following statements hold:
        \begin{itemize}
            \item $u(p) = u(q) = 0$.
            \item $u(p) > 0$ and $u(q) > 0$.
        \end{itemize}
    \item There exist an $A \subseteq \mathcal{P}(\{1,...,n\})$ such that 
    $u^{-1}(0) = \bigcup_{K \in A} \{ z \in [0,1]^n  \ | \ z_i = 0  \text{ if and only if } i \in K\} $.
\end{enumerate}
\end{prop}

\begin{proof}
    $(1 \rightarrow 2)$: We will show that $(\neg 2 \rightarrow \neg 1)$. By assumption there exist $p, q \in [0,1]^n$ such that $p_i = 0$ if and only if $q_i = 0$ and such that $u(p) = 0$ and $u(q) > 0$.  For each $i \in \{1,...n\}$ define the functions
    \begin{align*}
        \phi_i(y) = p_i \text{ for all }y \in Y \\
        \phi_i'(y) = q_i \text{ for all }y \in Y
    \end{align*}
    Clearly $\phi_i$ and $\phi_i'$ are continuous for all $i$ (in fact, they are constant functions) and $\phi_i \simeq \phi_i'$ for all $i$ , but
    \begin{align*}
        u(\phi_1,...,\phi_n)(y) = u(p) = 0 \\
        u(\phi'_1,...,\phi'_n)(y) = u(q) > 0
    \end{align*}
    which shows that $u(\phi_1,...,\phi_n) \not \simeq  u(\phi'_1,...,\phi'_n)$. 

    $(2 \rightarrow 3)$: Suppose (2) holds. For any $K \subseteq \{1,...,n\}$ we write \[ Z_K := \{ z \in [0,1]^n  \ | \ z_i = 0  \text{ if and only if } i \in K\}. \] We claim that for \[ A: = \{ K \subseteq \{1,...,n\} \ | Z_K \cap u^{-1}(0) \not = \emptyset \} \] we have $u^{-1}(0) = \bigcup_{K \in A} Z_K$. To see this, first assume that $x \in u^{-1}(0)$. Writing $L := \{i \in \{1,...,n\}: x_i = 0\}$, we have $x \in Z_L \cap u^{-1}(0)$, hence $L \in A$ and $x \in \bigcup_{K \in A} Z_K$. 

    Conversely, suppose that $x \in Z_K$ for some $K \in A$. By definition of $A$ there exists an $x' \in Z_K \cap u^{-1}(0)$. But then $x_i = 0$ if and only if $x'_i = 0$ for all $i \in \{1,...,n\}$ and $u(x') = 0$. So by condition (2), $u(x) = 0$, hence $x \in u^{-1}(0)$.
    
    $(3 \rightarrow 1)$: For each $i \in \{1,...,n\}$, let $\phi_i$ and $\phi'_i$ be isomorphic elements of $\mathbf{U}(Y)$ such that $\phi_i \simeq \phi_i'$. By $\Delta_i$ we denote the modulus of uniform continuity from $\phi_i \sq \phi_i'$ and by $\Delta_u$ we denote the modulus of uniform continuity for $u$. We will show
    \[ u(\phi_1,\ldots,\phi_n) \sq u(\phi_1',\ldots,\phi_n'). \]
    To that end, let $\varepsilon > 0$.

    Define \[ \gamma :=  \min ( \min( \Delta_i(\Delta_u(\varepsilon)) : i \in \{1,...,n\})), \Delta_u(\varepsilon)). \] Clearly $\gamma > 0$. We note that $u^{-1}(0)$ is closed since it is the preimage of a closed set under a continuous function. For any $z \in [0,1]^n$ we define the function $ d_n(\cdot,u^{-1}(0)): z \mapsto \inf(d_n(z,x) \ | \ x \in u^{-1}(0) )$. It is known that a function such as $d_n(\cdot,u^{-1}(0))$ is continuous if $u^{-1}(0)$ is closed, hence the function $z \mapsto d_n(z,u^{-1}(0))$ is continuous. But then \[ d_n(\cdot,u^{-1}(0))^{-1}([\gamma, 1]) \] is closed, hence compact. This implies that \[ \delta := \inf \{ u(z) \ | \ d_n(z,u^{-1}(0)) \geq \gamma\} =  \inf u(d_n(z,u^{-1}(0))^{-1}([\gamma, 1])) > 0. \] Indeed, if this were not the case, \reftheo{Extreme value theorem} would imply the existence of a point $p \in [0,1]^n$ such that $d_n(p, u^{-1}(0)) > \gamma $ but with $u(p) = 0$. This is clearly a contradiction.

    Now suppose $y \in Y$ is such that $u(\phi_1,...,\phi_n)(y) < \delta = \inf \{ u(z) \ | \ d_n(z,u^{-1}(0)) \geq \gamma\}$ then $d_n((\phi_1(y),...,\phi_n(y)),u^{-1}(0)) < \gamma$. This means that there exists some $p \in u^{-1}(0)$ such that $d_n((\phi_1(y),...,\phi_n(y)),p) \leq \gamma$.
    Therefore, for each $i \in \{1,...,n\}$ it holds that $|\phi_i(y) - p_i| \leq \gamma$.
    By assumption, $p \in \{ z \in [0,1]^n  \ | \ z_i = 0  \text{ if and only if } i \in K\}$ for some $K \in A$. So for all $j \in K$ it holds that $\phi_j(y) \leq \gamma$ hence $\phi'_j(y) \leq \Delta_u(\varepsilon)$ since $\phi_j \sq \phi'_j$. 
    
    Let $q$ be the point in $[0,1]^n$ such that $q_i = 0$ if $i \in K$ and $q_i = \max(\phi_i'(y),\Delta_u(\varepsilon))$ if $i \notin K$. Then for all $i \in \{1,...,n\}$ we have that $q_i = 0$ if and only if $p_i = 0$ so by assumption $u(q) = 0$. Next, we see that $d_n((\phi_1'(y),...,\phi_n'(y)),q) \leq \Delta_u(\varepsilon)$. By uniform continuity of $u$ we get $u(\phi_1'(y),...,\phi_n'(y)) =  |u(\phi_1'(y),...,\phi_n'(y)) - u(q)| \leq \varepsilon$. We conclude that $u(\phi_1,...,\phi_2) \sq u(\phi_1',...,\phi_n')$ and $u(\phi_1,...,\phi_2) \simeq u(\phi_1',...,\phi_n')$, by symmetry.
\end{proof}

To summarise,  if $u: [0,1]^n \rightarrow [0,1]$ is a uniformly continuous function for $n \geq 1$ that does not satisfy condition $2$ of \refprop{bad connectives for CMT} then there exists some signature $\Sigma$, for all $i \in \{1,...,n\}$ there exist some atomic formulas-in-context $(\phi_i|\bar x )$ and $(\phi_i'| \bar x )$ over $\Sigma$, and there exists an interpretation $\brac{\cdot}$ of $\Sigma$ in ${\mathbf{CMT}}$ such that $\brac{\phi_i'}_{\bar x} \simeq \brac{\phi_i'}_{\bar x}$ for all $i$ while $\brac{u(\phi_1,...,\phi_n) }_{\bar x} \not \simeq \brac{u(\phi_1',...,\phi_n') }_{\bar x}$.

Unfortunately, it seems that some functions considered in continuous model theory do violate condition $2$. As an example, consider the operation $u: [0,1] \to [0,1]: x \mapsto 1-x$, which is sometimes regarded as a kind of negation. It violates condition $2$, which can be seen by comparing $u(\frac{1}{2})$ and $u(1)$. This means that one can have predicates that are logically equivalent in the sense of $\mathbf{CMT}$, but no longer are after negating them in this sense. Clearly, this points to some limitation in our analysis of continuous model theory; how serious it is, is unclear to us.

\section{A category equivalent to complete metric spaces}

In this section we will show that the category of equivalence relations over $\mathbf{U}$ is equivalent to the category ${\bf cMet}_1$.

\subsection{The category of (partial) equivalence relations} Here we recall the construction of the category of (partial) equivalence relations over a coherent hyperdoctrine. For this subsection, fix a coherent hyperdoctrine $\mathbf{P}: \mathcal{C}^{op} \rightarrow {\mathbf{DistPrelattices}}$.

\begin{defi}{PERs}
 \textit{A partial equivalence relation over} $\mathbf{P}$ is a pair $(X,\sim)$ where $X$ is an object in $\mathcal{C}$ and $\sim$ is an element of $\mathbf{P}(X \times X)$, such that the following two sequents hold:
\begin{enumerate}
    \item  $x \sim y {\ \vdash_{xy} \ } y \sim x$ ($\sim$ is \textit{symmetric})
    \item  $x \sim y \wedge y \sim z {\ \vdash_{xyz} \ } x \sim z$ ($\sim$ is \textit{transitive})
\end{enumerate}
We can define the category ${\rm PER}(\mathbf{P})$ of (partial) equivalence relations over $\mathbf{P}$, where objects are (partial) equivalence relations $(X, \sim)$ over $\mathbf{P}$, and morphisms $(X,\sim) \rightarrow (Y,\sim)$ are equivalence classes of \textit{functional relations $F$}, which are elements of $\mathbf{P}(X \times Y)$ such that the following sequents hold:
\begin{enumerate}
\item  $F(x,y) {\ \vdash_{xy} \ } x \sim x \wedge y \sim y$ ($F$ is \textit{strict})
\item  $F(x,y) \wedge x \sim x' \wedge y \sim y' {\ \vdash_{xx'yy'} \ } F(x',y')$ ($F$ is \textit{relational})
\item  $F(x,y) \wedge F(x,y') {\ \vdash_{xyy'} \ } y \sim y'$ ($F$ is \textit{single-valued})
\item  $x \sim x {\ \vdash_x \ } \exists y \, F(x,y)$ ($F$ is \textit{total})
\end{enumerate}
Two such functional relations $F, G \in \mathbf{P}(X \times Y)$ will be considered equivalent if
\[ F(x,y) \vdash_{xy} G(x,y) \]
holds (from which $G(x,y) \vdash_{xy} F(x,y)$ follows).

For morphisms $f: (X,\sim) \rightarrow (Y,\sim)$ and $g: (Y,\sim) \rightarrow (Z,\sim)$ which are represented by functional relations $F$ and $G$, respectively, we define the composition $f \circ g$ to be the isomorphism class of the functional relation $\brac{\exists y \, (F(x,y) \wedge G(y,z))}_{xz}$.
\end{defi}

\begin{defi}{ERs}
An \emph{equivalence relation over $\mathbf{P}$} is a partial equivalence relation $(X,\sim)$ over $\mathbf{P}$ for which $\sim$ is also reflexive; that is, the following sequent is satisfied: 
\[ \top {\ \vdash_{x} \ } x \sim x. \]
We write ${\rm ER}(\mathbf{P})$ for the full subcategory of ${\rm PER}(\mathbf{P})$ on the equivalence relations. Note that if $(X,\sim)$ and $(Y, \sim)$ are equivalence relations, any relation $F \in \mathbf{P}(X \times Y)$ is automatically strict and the totality condition can be simplified to $\vdash_x \, \exists y \, F(x,y)$.
\end{defi}

\subsection{From pseudometric spaces to equivalence relations}
The next step in our analysis is to construct a functor $G: \mathbf{pMet}_1 \to {\rm ER}(\mathbf{U})$.

Let $(X,d)$ be an object in $\mathbf{pMet}_1$. It can also be considered as an object in ${\rm ER}(\mathbf{U})$: firstly, the diameter of $(X,d)$ is by definition less than or equal to 1, so $d \in \mathbf{U}(X \times X)$. Secondly, $d$ is clearly a reflexive and symmetric element of $\mathbf{U}(X \times X)$. Lastly, suppose $\varepsilon > 0$ and $\max ( d(x,y), d(y,z)) \leq \frac{\varepsilon}{2}$. It follows that $d(x,z) \leq d(x,y) + d(y,z) \leq \varepsilon$ by the triangle inequality, hence that $d$ is transitive.

If $f: (X,d) \rightarrow (Y,d)$ is a uniformly continuous function between pseudometric spaces with modulus of uniform continuity $\Delta$, consider the element in $\mathbf{U}(X \times Y)$ that takes a pair $(x,y)$ in $X \times Y$ to $ d(f(x),y)$. We will show that this element is relational, single-valued and total.

Firstly, the sequent $\left(d(f(x),y) \wedge d(x,x') \wedge d(y,y')) \vdash_{xx'yy'} d(f(x'),y') \right)$ holds in $\mathbf{U}$: If $(x, x', y, y') \in X \times X \times Y \times Y$ with $\max( d(f(x),y), d(x,x'),  d(y,y') ) \leq \min( \Delta(\frac{1}{3}\varepsilon), \frac{1}{3}\varepsilon)$ then $d(f(x'),y') \leq d(f(x'),f(x)) + d(f(x),y) + d(y,y') \leq \varepsilon$ by the triangle inequality. 

Secondly, the sequent $(d(f(x),y) \wedge d(f(x),y')) \vdash_{xyy'} d(y,y')$ holds in $\mathbf{U}$ since $\max(d(f(x),y) \wedge d(f(x),y')) \leq \frac{1}{2}\varepsilon$ implies that $d(y,y') \leq d(y,f(x)) + d(y',f(x)) \leq \varepsilon$ for all $(x,y,y') \in X \times Y \times Y$.

Lastly, showing the validity of the sequent $\top \vdash_{x} \exists y (d(f(x),y))$ amounts to showing that for any $x \in X$ the infimum ${\rm inf}\{ d(f(x), y) \, : \, y \in Y\} = 0$, which can be done by choosing $y = f(x)$.

The assignments above define a functor:
\begin{prop}{Functor G definition}
    For any object $(X,d)$ in $\mathbf{pMet}_1$ let $G(X,d)$ be $(X,d)$ considered as an object in ${\rm ER}(\mathbf{U})$. If $f: (X,d) \rightarrow (Y,d)$ is a morphism in $\mathbf{pMet}_1$ then let $G(f)$ be the class of functional relations in $\mathbf{U}(X \times Y)$ that are isomorphic to the functional relation $(x,y) \mapsto d(f(x),y)$. This defines a functor $G: \mathbf{pMet}_1 \rightarrow {\rm ER}(\mathbf{U})$. 
\end{prop}

\begin{proof}
If $\id: (X,d) \rightarrow (X,d)$ is the identity in $\mathbf{pMet}_1$, then $G(\id)$ is represented by the functional relation $(x,x') \mapsto d(x,\id(x')) = d(x,x')$ which represents the identity of $G(X,d)$ in ${\rm ER}(\mathbf{U})$.
For functoriality, suppose that
    \begin{equation*}
        \begin{tikzcd}
{(X,d)} \arrow[r, "f"] & {(Y,d)} \arrow[r, "g"] & {(Z,d)}
\end{tikzcd}
    \end{equation*}
    is a diagram in $\mathbf{pMet}_1$. Then $(x,y) \mapsto d(f(x),y)$ is the functional relation that represents $G(f)$. Similarly, $(x,z) \mapsto d(g(x),z)$ represents $G(g)$ and $(x,z) \mapsto d(gf(x),z)$ represents $G(gf)$. So to show that $G(g) \circ G(f) = G(gf)$ it will suffice to show that $$d(gf(x),z)  \vdash_{xz} \exists y(d(f(x),y) \wedge d(g(y),z) ) $$ holds in $\mathbf{U}$. But note that if $d(gf(x),z) \leq \varepsilon$, then
    \begin{eqnarray*} {\rm inf} \{ {\rm max}(d(f(x),y),d(g(y),z)) \, : \, y \in Y \} & \leq & {\rm max}(d(f(x),f(x)),d(g(f(x)),z)) \\ & = & d(g(f(x)),z) \\
        & \leq & \varepsilon,
    \end{eqnarray*}
    which shows the desired sequent.
\end{proof}

\begin{rema}{GalsoforcMet}
    A similar functor as in \refprop{Functor G definition} can be defined with domain $\mathbf{Met}_1$ or $\mathbf{cMet}_1$. We also denote these functors by $G$.
\end{rema}

\subsection{An equivalence of categories} Finally, we show that $G$ becomes an equivalence when we restrict the domain to $\mathbf{cMet}_1$.

\begin{prop}{Gfaithful}
The functor $G: \mathbf{Met}_1 \rightarrow {\rm ER}(\mathbf{U})$ faithful. 
\end{prop}

\begin{proof}
    Let $f,g$ be uniformly continuous functions $(X,d) \rightarrow (Y,d)$ in $\mathbf{Met}_1$. Then $Gf = Gg$ implies that $
    \brac{d(f(x),y)}_{xy} \simeq \brac{d(g(x),y)}_{xy}$ as relations in $\mathbf{U}(X\times Y)$. In particular, for all $x \in X$ and $y \in Y$ we have that $d(f(x),y) = 0$ if and only if $d(g(x),y) = 0$. So let $x \in X$ and $y := f(x)$. Then $0 = d(f(x),f(x)) = d(g(x),f(x))$. Since $(Y,d)$ is a metric space, $f(x) = g(x)$ for all $x \in X$. Therefore $f = g$.
\end{proof}

\begin{prop}{G full and faithful}
The functor $G: \mathbf{cMet}_1 \rightarrow {\rm ER}(\mathbf{U})$ is full and faithful.
\label{prop: G fully faithful}
\end{prop}

\begin{proof}
Let $F: (X,d) \rightarrow (Y,d)$ represent a morphism in ${\rm ER}(\mathbf{U})$. Let $x \in X$. From the totality of $F$ it follows that $\inf_{y \in Y}F(x,y) = 0$, so there exists a sequence $(y_n)_{n \in \N}$ in $Y$ such that $F(x,y_n) \rightarrow 0$. Let $(y_n)_{n \in \N}$ and $(z_n)_{n \in \N}$ be any two of such sequences. For any $\varepsilon \gt 0$ there exists a $\delta > 0$ such that for all $y,y' \in Y$ we have that $F(x,y) \wedge F(x,y') < \delta$ implies that $d(y,y') < \varepsilon$. It follows that the mixed sequence $(y_1, z_1, y_2, z_2, ...)$ is Cauchy. By completeness of $(Y,d)$ the mixed sequence converges to some element $y'$ in $Y$, hence the sequences $(y_n)_{n \in \N}$ and $(z_n)_{n \in \N}$ both converge to $y'$. We can define the function $f: X \rightarrow Y$ by setting $f(x) = y'$. 

If $\varepsilon \gt 0$ then we may use that $F$ is relational to find a $\delta$ such that for all $y \in Y$ ${\rm max}(F(x,y), d(y,f(x))) \leq \delta$ implies $F(x,f(x)) \leq \varepsilon$. By the above, there exists a $y_n$ such that $F(x,y_n) \leq \delta$ and $d(y_n,f(x)) \leq \delta$. This implies that $F(x,f(x)) \leq \varepsilon$ for any $\varepsilon \gt 0$ and hence that $F(x,f(x)) = 0$.

We can show the function $f$ is uniformly continuous from $(X,d)$ to $(Y,d)$: Let $\varepsilon > 0$. From the previous observation and the fact that $F$ is single-valued it follows that there exists a $\delta_1 > 0$ such that for all $x, x' \in X$ with $F(x,f(x')) < \delta_1$ we have $d(f(x),f(x')) < \varepsilon$. Moreover, since $F$ is relational, there exists a $\delta_2 > 0$ such that for any $x, x' \in X$ with $d(x,x') < \delta_2$ we have $F(x,f(x')) < \delta_1$. So if $d(x,x') < \delta_2$ then it follows that $F(x,f(x')) < \delta_1$ and hence that $d(f(x),f(x')) < \varepsilon$. Therefore $f$ is uniformly continuous. 

Finally, it remains to show that $Gf = [F]$, for which it suffices to show that
\[ F(x,y) \vdash_{x,y} d(f(x),y). \] 
But this follows from the fact that $F$ is single-valued and $F(x,f(x)) = 0$.
\end{proof}

\begin{prop}{Giso}
Let $\phi: (X,d) \rightarrow (\bar X, \bar d)$ be the metric space completion of an object $(X,d)$ in $\mathbf{pMet}_1$. The morphism $G(\phi)$ is an isomorphism in ${\rm ER}(\mathbf{U})$.
\label{prop: metric completion is an isomorphism}
\end{prop}

\begin{proof}
Recall that $\bar X$ is the set of equivalence classes of Cauchy sequences in $(X,d)$ where $(x_n)_{n \in \N} \sim (y_n)_{n \in \N}$ if $\lim_n d(x_n,y_n) = 0$, and \[ \bar d((x_n)_{n \in \N},(y_n)_{n \in \N}) = \lim_n d(x_n,y_n) \] and $\phi(x)$ is the constant sequence. Let $\varepsilon > 0$ and $(x_n)_{n \in \N} \in \bar X$. There exists $K \geq 0$ such that $d(x_n,x_m) \leq \varepsilon $ for all $m , n \geq K$. Then $\bar d(\phi(x_K),(x_n)_n) = \lim_n d(x_K,x_n) \leq \varepsilon$, hence $\inf_{x \in X} \bar d (\phi(x),(x_n)_n) = 0$ and the relation $\bar d(x',\phi(x))$ on $\mathbf{U}(\bar X \times X)$ is total. That it is single-valued is trivial. Since $G(\phi)$ is represented by the same relation on $\mathbf{U}(X \times \bar X)$, $G(\phi)$ is an isomorphism. 
\end{proof}

\begin{theo}{}
The functor $G: \mathbf{cMet}_1 \rightarrow {\rm ER}(\mathbf{U})$ \text{is an equivalence of categories}
\end{theo}

\begin{proof}
We will show, firstly, that the functor $G: \mathbf{pMet}_1 \rightarrow {\rm ER}(\mathbf{U})$ is essentially surjective. Let $(X,R)$ be an object in ${\rm ER}(\mathbf{U})$.
 Define $U_a = \{ (x,y) \in X \times X: R(x,y) \leq a \text{ with }a \in (0,1]  \cap \mathbb{Q} \}$ 
 and define $\Psi = \{V: \exists a \in (0,1] \cap \mathbb{Q} \text{ s.t. } U_a \subseteq V\}$. The collection $\Psi$ is a uniformity: 
\begin{enumerate}
    \item Clearly, every member $\Psi$ contains the diagonal.
    \item If $U \in \Psi$, then $U_a \subseteq U$ for some $a \in (0,1]  \cap \mathbb{Q}$. From the symmetry of $R$, it follows that $U_b \subseteq U_a^{-1}$ for some $b \in (0,1]  \cap \mathbb{Q}$, so $U_a^{-1} \in \Psi$ hence $U^{-1} \in \Psi$.
	\item If $U \in \Psi$, then $U_a \subseteq U$ for some $a \in (0,1]  \cap \mathbb{Q}$. From the transitivity of $R$, it follows that $U_b \circ U_b \subseteq U_a \subseteq U$ for some $b \in (0,1]  \cap \mathbb{Q}$
	\item If $U, V \in \Psi$, then $U_a \subseteq U$ and $U_b \subseteq V$ with $a,b \in (0,1]  \cap \mathbb{Q}$. Then $U_{min(a,b)} \subseteq U_a \cap U_b \subseteq U \cap V$, so $U \cap V \in \Psi$.
	\item Clearly, if $U \in \Psi$ and $U \subseteq V$ then $V \in \Psi$.
\end{enumerate}

So, by Theorem $\ref{uniformity induced by metric}$, there exists a metric $d$ on $X$ inducing the uniformity $\Psi$, meaning that 
$\Psi = \{ V: \exists b \in (0,1] \text{ s.t. } V_b \subseteq V\}$ where $V_b = \{(x,y) \in X \times X: d(x,y) \leq b\}$.
Furthermore, $d$ represents an isomorphism $(X,d) \rightarrow (X,R)$ in ${\rm ER}(\mathbf{U})$:
\begin{enumerate}
\item Let $\varepsilon > 0$. Since $d$ induces $\Psi$, there exists a $\delta>0$ such that $U_{\delta} \subseteq V_{\frac{1}{3}\varepsilon}$. If $x,x',y$ and $y'$ are points in $X$ such that $\max(d(x,y),R(y,y'),d(x,x')) \leq \min(\delta,\frac{1}{3}\varepsilon)$, then $d(x',y') \leq \varepsilon$, by the triangle inequality. Therefore, $d$ is relational.
\item To show that $d$ is single-valued, let $\varepsilon > 0$. Since $d$ induces $\Psi$, there exists a $\gamma > 0$ such that $V_\gamma \subseteq U_\varepsilon$. Let $\delta = \frac{1}{2}\gamma$. If $x,y$ and $y'$ are elements of $X$ such that $\max(d(x,y),d(x,y')) \leq \delta$, then $(y,y') \in V_\gamma$, hence $R(y,y') \leq \varepsilon$. A similar argument shows that the converse of $d$ is single-valued as well.
\item Clearly, $d$ is total, and the converse holds as well.
\end{enumerate}
So $G: \mathbf{pMet}_1 \rightarrow {\rm ER}(\mathbf{U})$ is essentially surjective. This fact, together with Proposition $\ref{prop: metric completion is an isomorphism}$ and Proposition $\ref{prop: G fully faithful}$, shows that $G: \mathbf{cMet}_1 \rightarrow {\rm ER}(\mathbf{U})$ is an equivalence.
\end{proof}

\begin{rema}{functortounifsp}
    The argument is the theorem above can be extended to show that there is a functor from the category ${\rm ER}(\mathbf{U})$ to the category of uniform spaces. It may be interesting to study this functor further: for instance, is it an embedding?
\end{rema}

\section{A coherent category for continuous logic}

We return to the task of embedding the hyperdoctrine $\mathbf{CMT}$ into the subobject hyperdoctrine of a topos. In this section we take the first step in that direction: we will embed $\mathbf{CMT}$ into the subobject hyperdoctrine of a coherent category. The hyperdoctrine that we will consider is the category of partial equivalence relations (PERs) over the hyperdoctrine $\mathbf{U}$. By characterising the subobjects in categories of partial equivalence relations as strict relations, we will be able to define an embedding from $\mathbf{CMT}$ into the subobject hyperdoctrine of ${\rm PER}(\mathbf{U})$.

\subsection{The category of partial equivalence relations over a coherent hyperdoctrine} \label{section: category of partial equivalence relations} Our first goal will be to establish that ${\rm PER}(\bfp)$ is a coherent category whenever $\bfp$ is a coherent hyperdoctrine.

\begin{lemm}{PERPexact}
    If $\mathbf{P}$ is a coherent hyperdoctrine, the category ${\rm PER}(\bfp)$ is exact.
\end{lemm}
\begin{proof}
    Since a coherent hyperdoctrine is in particular an ``elementary existential doctrine'', Corollary 3.4 in \cite{maietti_rosolini_2013} implies that ${\rm PER}(\bfp)$ is exact.
\end{proof}

To show that ${\rm PER}(\bfp)$ is coherent, we will take a closer look at the hyperdoctrine of strict relations.

\begin{defi}{Strict relation} If $\mathbf{P}$ is a coherent hyperdoctrine, let Strict($\mathbf{P}$) be the indexed preorder defined as follows. Its base is ${\rm PER}(\mathbf{P})$, while its predicates on an object $(X,\sim)$ in ${\rm PER}(\mathbf{P})$ are the strict relations on $(X, \sim)$: that is, elements $\phi$ of $\mathbf{P}(X)$ such that
    \begin{enumerate}
        \item $\phi(x) {\ \vdash \ } x \sim x$
        \item $\phi(x) \wedge x \sim x' {\ \vdash \ } \phi(x')$
    \end{enumerate}
These strict relations inherit a preorder structure from $\mathbf{P}(X)$. Finally, if \[ f: (X, \sim) \to (Y, \sim) \] is a morphism in ${\rm PER}(\mathbf{P})$ represented by $F \in \mathbf{P}(X \times Y)$ and $\phi$ is a strict relation on $(Y, \sim)$, then $\exists y \, \big( \, F(x,y) \land \phi(y)   \, \big)$ defines a strict relation on $(X, \sim)$.
\end{defi}

\begin{prop}{stricthyperdoctrine}
    If $\mathbf{P}$ is a coherent hyperdoctrine, the indexed preorder ${\rm Strict}(\mathbf{P})$ is as well.
\end{prop}
\begin{proof}
    Note that if $\phi$ and $\psi$ are strict relations on $(X, \sim)$, then so are $\phi \land \psi$ and $\phi \lor \psi$ with meet and join computed as in $\mathbf{P}(X)$. The bottom element of $\mathbf{P}(X)$ is a strict relation on $(X, \sim)$ and the bottom element in the poset of strict relations; the top element in ${\rm Strict}(\mathbf{P})(X, \sim)$ is given by the predicate $x \sim x$. All of this structure is clearly preserved by the pullback operation.

    Finally, if $f: X \to Y$ is a morphism ${\rm PER}(\mathbf{P})$ represented by $F \in \mathbf{P}(X \times Y)$ and $\psi$ is a strict relation on $(X, \sim)$, then $\exists x \, \big( \, F(x,y) \land \psi(x) \, \big)$ defines a strict relation on $(Y, \sim)$. The reader can quickly check that this defines a left adjoint to pullback along $f$ for which Frobenius is satisfied.

    For verifying that Beck-Chevalley is satisfied, consider a pullback diagram
    \begin{equation}
        \begin{tikzcd}
            (X \times Y, \sim_P) \ar[r, "{[P]}"] \ar[d, "{[Q]}"] & (X, \sim) \ar[d, "{[F]}"] \\
            (Y, \sim) \ar[r, "{[G]}"] & (Z, \sim)
        \end{tikzcd}
    \end{equation}
    where
    \begin{eqnarray*}
        (x,y) \sim_P (x',y') & := & x \sim x' \land y \sim y' \land \exists z \, ( \, F(x,z) \land G(y,z) , ), \\
        P((x,y), x') & := & x \sim x' \land (x,y) \sim_P (x,y), \\
        Q((x,y), y') & := & y \sim y' \land (x,y) \sim_P (x,y).
    \end{eqnarray*}
    In this case we can calculate
    \begin{eqnarray*}
        (\exists_{[P]} [Q]^* \phi)(x) & \simeq & \exists (x',y') \, ( \, P((x',y'),x) \land [Q]^*\phi(x',y') \, ) \\
        & \simeq & \exists (x',y') \, ( \, P((x',y'),x) \land \exists y \, ( \, Q((x',y'), y) \land \phi(y) \, ) \, ) \\
        & \simeq & \exists y \, ( \, (x, y) \sim_P (x,y) \land \phi(y) \, ) \\
        & \simeq & \exists y, z ( \, F(x,z) \land G(y, z) \land \phi(y) \, ) \\
        & \simeq & \exists z \, ( \, F(x,y) \land \exists y \, (\phi(y) \land G(y,z)) \, ) \\
        & \simeq & \exists_{[G]} [F]^* \phi)(x),
    \end{eqnarray*}  
    showing that the Beck-Chevalley condition holds.  
\end{proof}

\begin{prop}{Subobject correspondence in PER} If $\mathbf{P}$ is a coherent hyperdoctrine, then its subobject hyperdoctrine is isomorphic to ${\rm Strict}(\mathbf{P})$; in particular, the subobject hyperdoctrine is coherent.
\end{prop}
\begin{proof}
    We use that a morphism $f: (Y, \sim) \to (X, \sim)$ represented by $F \in \bfp(Y \times X)$ is monic in ${\rm PER}$($\bfp$) precisely when
    \[ F(y, x), F(y',x) \vdash y \sim y' \]
    holds in the hyperdoctrine $\bfp$ (see Corollary 2.9 of \cite{hyland_johnstone_pitts_1980}).
    So if $\phi \in {\rm Strict}(\mathbf{P})(X, \sim)$ is a strict relation, then $[\sim_\phi]: (X, \sim_\phi) \to (X, \sim)$ is monic, where
    \begin{eqnarray*}
        x \sim_\phi x' & :\Leftrightarrow & \phi(x) \land x \sim x';
    \end{eqnarray*}
    conversely, if $[F]: (Y, \sim) \to (X, \sim)$ monic, then 
    \[ \psi(x) :\Leftrightarrow \exists y \, F(y,x) \]
    defines a strict relation on $(X, \sim)$. We leave the verification that these operations are monotonic, compatible with the pullback operations, and each other's inverses to the reader.
\end{proof}    

\begin{coro}{PER(P)  and sub and strict are coh}
If $\mathbf{P}$ is a coherent hyperdoctrine, the category ${\rm PER}(\bfp)$ is coherent.
\end{coro}
\begin{proof}
Immediate in view of \reflemm{PERPexact}  and \refprop{Subobject correspondence in PER}.
\end{proof}

If $\bfp$ is a universally coherent hyperdoctrine, then we would expect that ${\rm Strict}(\bfp)$ has a right adjoint, for a projection $\pi: A \times B \rightarrow A$ in ${\rm PER}(\bfp)$. This is true if $B$ is an equivalence relation:

\begin{prop}{General: P has right adjoints}
Let $\mathbf{P}$ be a universally coherent hyperdoctrine{{{}}}, and suppose that $\Pi: (Y, \sim) \times (Z,\sim) \rightarrow (Y,\sim)$ is a projection in ${\rm PER}(\mathbf{P})$, such that the sequent $z \sim z {\ \dashv \hspace{1pt} \vdash \ } \top$ holds in $\bfp$. Then there exists a functor of preorders \[ \forall_{\Pi}: {\rm Strict}(\bfp)(Y \times Z, \sim) \rightarrow {\rm Strict}(\mathbf{P})(Y,\sim) \] such that $\Pi^* \dashv \forall_{\Pi}$.  
\end{prop}

\begin{proof}
    Define $$\forall_\Pi(\phi(y,z)) := y \sim y \wedge \forall z \, \phi(y,z)$$ for any strict relation $\phi$ on $(Y\times Z, \sim)$. The interpretation of $(y \sim y \wedge \forall z P(y,z))$ is clearly a strict relation on $(Y,\sim)$, and the assignment clearly preserves the ordering.
    We have an adjunction: Let $P$ be a strict relation on $(Y \times Z, \sim)$, and let $Q$ be a strict relation on $(Y, \sim)$. The following holds in $\bfp$.
    \begin{align*}
        &\exists y'(Q(y') \wedge y' \sim y) {\ \vdash_{yz} \ } P(y,z) \\
        \iff \quad &Q(y) \wedge y \sim y {\ \vdash_{yz} \ } P(y,z) \\
        \iff \quad &Q(y) \wedge y \sim y {\ \vdash_{y} \ } \forall z' P(y,z') \\
        \iff \quad &Q(y) {\ \vdash_{y} \ } y \sim y \wedge \forall z' P(y,z') 
    \end{align*}
    Here we use that $Q$ is strict.
\end{proof}

\begin{rema}{PERUpretopos} The category ${\rm PER}(\bfp)$ is actually a pretopos. A proof can be found in the first author's MSc thesis.
\end{rema}

\subsection{Embedding $\mathbf{CMT}$} The equivalence 
\[ G: {\mathbf{cMet}}_1 \to {\rm ER}(\mathbf{U}) \]
extends to an embedding ${\mathbf{cMet}}_1 \to {\rm PER}(\mathbf{U})$ which we will also denote by $G$. In this subsection we will show that this embedding $G$ extends to an embedding of hyperdoctrines.

\begin{prop}{Gpreserve}
The functor $G: \mathbf{cMet}_1 \rightarrow {\rm PER}(\mathbf{U})$ preserves finite products.
\label{G preserves products}
\end{prop}

\begin{proof}
    $G$ preserves the terminal object since $\mathbf{U}(G(1))$ contains only one element and any nonempty set $X$ together with the top element  of $\mathbf{U}(X \times X)$ is a terminal object in ${\rm PER}(\mathbf{U})$.
    
    For products, let \begin{equation*}
    \begin{tikzcd}
    {(X,d_X)} & {(X \times Y,d)} \arrow[l, "\pi_X"'] \arrow[r, "\pi_Y"] & {(Y,d_Y)}
    \end{tikzcd}
    \end{equation*}
    be a product diagram in $\mathbf{pMet}_1$. The product of $G(X,d_X)$ and $G(Y,d_Y)$ in ${\rm PER}(\mathbf{U})$ is the object $(X \times Y,\approx)$ where $X \times Y$ is the set-theoretic product of $X$ and $Y$ and $\approx$ is the relation on $\mathbf{U}(X \times Y \times X \times Y)$ such that 
    $$(x,y) \approx (x',y') := d_X(x,x') \wedge d_Y(y,y')$$
   By definition, $d$ is the maximum of $d_X$ and $d_Y$, and we have shown that the meet in $\mathbf{U}$ is given by the maximum, so clearly $(X \times Y,\approx)$ is equal to $G((X,d_X) \times (Y,d_Y))$. If we denote the projection $$(X\times Y,\approx) \rightarrow G(X,d)$$ in ${\rm PER}(\mathbf{U})$ by $\Pi_X$ then the following sequent holds in $\mathbf{U}$.
    \begin{align*}
        \Pi_X((x,y),x') &{\ \dashv \hspace{1pt} \vdash \ } d_X(x,x') \wedge d_Y(y,y) \\
        &{\ \dashv \hspace{1pt} \vdash \ } d_X(x,x') \\
        &{\ \dashv \hspace{1pt} \vdash \ } G( \pi_X)((x,y),x')
    \end{align*}
    \end{proof}

\begin{prop}{embCMTinsubjPERU}
    The product-preserving embedding $G: {\mathbf{cMet}}_1 \to {\rm PER}(\mathbf{U})$ extends to an embedding of hyperdoctrines ${\mathbf{CMT}} \to ({\rm PER}(\mathbf{U}), {\rm Sub})$; this embedding preserves all the structure of a universally coherent hyperdoctrine.
\end{prop}
\begin{proof}
    In view of \refprop{Subobject correspondence in PER} it suffices to construct an embedding ${\mathbf{CMT}} \to {\rm Strict}(\mathbf{U})$. On objects, a metric space $(X,d)$ is sent to an element $(X, \sim)$ with $x \sim x' = d(x,x')$. An element $\phi \in {\mathbf{CMT}}(X,d)$ is a uniformly continuous function $\phi: (X,d) \to [0,1]$; such an element can be thought of an element $\phi \in \mathbf{U}(X)$ as well. It will be a strict relation on $(X, \sim)$, because of uniform continuity of $\phi$. We only have to verify that the logical structure (distributive lattice structure on the fibres as well as the adjoints) is preserved: however, on reflexive objects the logical structure  in ${\rm Strict}(\mathbf{U})$ is computed as in $\mathbf{U}$ (see the proofs of \refprop{stricthyperdoctrine} and \refprop{General: P has right adjoints}). Moreover, the logical structure of $\mathbf{CMT}$ and $\mathbf{U}$ is also computed in the same manner (see the proof of \refprop{CMThyperdoctrine}).
\end{proof}

\section{A topos for continuous logic}

In the previous section we have seen how the hyperdoctrine $\mathbf{CMT}$ can be embedded in the subobject hyperdoctrine of a coherent category. From this it follows fairly quickly, using some known facts about coherent topologies, that $\mathbf{CMT}$ can also be embedded into the subobject hyperdoctrine of a Grothendieck topos. We will recall these facts here and explain how these can be used to achieve our main objective. 

\subsection{Coherent embedding in sheaves}

Let $\mathcal{C}$ be a small coherent category. We can define the \textit{coherent Grothendieck topology} (also called the \textit{coherent topology}) $J$ on $\mathcal{C}$: If $X$ is an object in $\mathcal{C}$ we let the covering sieves $J(X)$ consist of all the finite sets of morphisms $(f_i: X_i \rightarrow X \ | \ i \in I)$ such that the union of the images of $f_i$ is $X$. 

It is known (see example A2.1.11 in \cite{johnstone_2002}) that for a small coherent category $\mathcal{C}$ the topology $J$ is \textit{subcanonical}, meaning that all representable presheaves on $\mathcal{C}$ are sheaves for $J$. In that case we can consider the Yoneda embedding as an embedding into the category ${\rm Sh}(\mathcal{C},J)$. This embedding preserves any universal quantification which exists, as made precise in the following lemma, which is a modified version of Lemma 3.1 in \cite{butz_johnstone_1998}. The obvious modification of their proof suffices to prove our case.

\begin{lemm}{small Heyting category: yoneda preserves universal quantification}{\rm  [Modification of Lemma 3.1 in \cite{butz_johnstone_1998}]}
    Let $\mathcal{C}$ be a small coherent category and let $J$ be a subcanonical Grothendieck topology on $\mathcal{C}$. Suppose that $f: A \rightarrow B$ is a morphism in $\mathcal{C}$ such that $Sub_{\mathcal{C}}(f)$ has a right adjoint $\forall_f$. Then the Yoneda embedding $y: \mathcal{C} \rightarrow {\rm Sh}(\mathcal{C},J)$ preserves this right adjoint.
\end{lemm}

We can do the same for Corollary 3.2 in \cite{butz_johnstone_1998}: 

\begin{lemm}{coherent small category means yoneda embedding is coherent functor}{\rm [Modification of Corollary 3.2 in \cite{butz_johnstone_1998}]}
If $\mathcal{C}$ is a small coherent category and $J$ denotes the coherent topology on $\mathcal{C}$, then the Yoneda embedding $\mathcal{C} \rightarrow {\rm Sh}(\mathcal{C},J)$ is a coherent functor.
\end{lemm}

 \subsection{Embedding into a topos of sheaves} \label{section: topos embedding}

We cannot use \reflemm{coherent small category means yoneda embedding is coherent functor} and Lemma \reflemm{small Heyting category: yoneda preserves universal quantification} directly to show an embedding of the subobject hyperdoctrine of ${\rm PER}(\mathbf{U})$ into that of a sheaf topos, since ${\rm PER}(\mathbf{U})$ is clearly not small. However, using Grothendieck universes we can obtain ``small versions'' of the constructions of the previous chapters. 

Indeed, let us fix some Grothendieck universe containing the unit interval $[0,1]$ and refer to the sets belonging to this universe as \emph{tiny}. Then let ${\mathbf{cMet}}^t_1$ be the category of those complete metric spaces $(X, d)$ of diameter 1 where $X$ is tiny; let ${\mathbf{CMT}}_t$ be the hyperdoctrine obtained from ${\mathbf{CMT}}$ by only including those metric spaces in the base that belong to ${\mathbf{cMet}}^t_1$; and let $\mathbf{U}^t$ be the hyperdoctrine obtained from $\mathbf{U}$ by only including the tiny sets in the base. By tracing through the proof of 
\refprop{embCMTinsubjPERU} it is clear that the same arguments show the following result as well:

\begin{prop}{embCMTinsubjPERUagain} 
    The embedding $G: {\mathbf{cMet}}^t_1 \to {\rm PER}(\mathbf{U}^t)$ extends to an embedding of hyperdoctrines ${\mathbf{CMT}}^t \to ({\rm PER}(\mathbf{U}^t), {\rm Sub})$; this embedding preserves all the structure of a universally coherent hyperdoctrine.
\end{prop}

\begin{theo}{CMT embeds into a topos} The hyperdoctrine
${\mathbf{CMT}}^t$ embeds into the subobject hyperdoctrine of a Grothendieck topos.
\end{theo}
\begin{proof}
This follows from \reflemm{small Heyting category: yoneda preserves universal quantification}, \reflemm{coherent small category means yoneda embedding is coherent functor} and \refprop{embCMTinsubjPERUagain}.
\end{proof}

\section{Questions and directions for future research}\label{section: conclusion}

We have introduced a hyperdoctrine $\mathbf{CMT}$ for continuous model theory and embedded it in the subobject hyperdoctrine of a suitable Grothendieck topos. In this way we have opened up a new way of thinking about continuous model theory, which should be explored further. Let us mention some possible directions.

Any Grothendieck topos has a rich logical structure: not only does it have an implication which $\mathbf{CMT}$ lacks, but it is a model of full higher-order intuitionistic logic. It is tempting to think that this can be exploited for the purposes of continuous model theory, but at present it is not entirely clear to us how. A good starting point may be to use Proposition 7.14 from \cite{benyaacovetal08} (cited above) to link the interpretation of formulas in the topos with their interpretation in $\omega$-saturated models in continuous model theory.

In general, the topos deserves a much closer study. For instance, is there a simple description of its subobject classifier? Another natural question in this regard is whether the topos satisfies any continuity principles, as it shares many features which a topos defined on a site of topological spaces (see Section VI.9 of \cite{maclanemoerdijk94}). 

Also, we have two objects in the the topos which can play the role of the unit interval. On the one hand, we can construct the unit interval internally, by using some standard construction (as Dedekind cuts, say). On the other hand, ${\mathbf{cMet}}_1$ embeds into the topos, so we also have the sheaf which is the image of the unit interval under this embedding. Are they the same? 

In this connection, it would also be interesting to explore variations on $\mathbf{U}$, suggested by strengthenings of uniform continuity. We mention two natural candidates: for two maps $\alpha, \beta: X \to [0,1]$, we could also order them by saying:
\[ \alpha \sq \beta \Leftrightarrow (\exists K > 0) \, (\forall x \in X) \, \beta(x) \leq K \alpha(x), \]
as in the definition of Lipschitz continuity. Presumably, one could again study the category of equivalence relations in this hyperdoctrine and embed it in a Grothendieck topos. The category of complete metric spaces and Lipschitz continuous maps embeds into the category of equivalence relations for this hyperdoctrine; perhaps it is again an equivalence? One could also restrict the hyperdoctrine even further and demand that $K = 1$; this should be related to the category of ultrametric spaces and 1-Lipschitz functions. 

\printbibliography

\appendix

\end{document}